%% file: writeup.tex
\documentclass[12pt]{article} 

\usepackage{graphicx}
\usepackage{cite}
\usepackage{amsmath}
\usepackage{amssymb}
\usepackage{amsthm}
\usepackage{subcaption}
\usepackage{enumerate}

\usepackage[hmargin=1 in, vmargin= 1.25 in]{geometry}

\title{ Generating Sets of \\ Mathieu Groups }

\author{Thomas G. Brooks}

\date{\today}

\newtheorem{theorem}{Theorem}[section]

\newtheorem{proposition}[theorem]{Proposition}
\newtheorem{corollary}[theorem]{Corollary}

\newenvironment{definition}[1][Definition]{\begin{trivlist}
\item[\hskip \labelsep {\bfseries #1}]}{\end{trivlist}}

\newcommand{\Z}{\mathbb{Z}}
\newcommand{\F}{\mathbb{F}}

\begin{document}

\maketitle

\begin{abstract}
    Julius Whiston\cite{WHISTON} calculated the maximum size of an irredundant generating set for $S_n$ and $A_n$ by examination of maximal subgroups.
    Using analogous considerations, we will compute upper bounds to this value for the first two Mathieu groups, $M_{11}$ and $M_{12}$.
    Computational results gave explicit irredundant generating sets of $M_{11}$ and $M_{12}$ of size 5 and 6, respectively.
    Together these give the full results that the maximum size of an irredundant generating set for $M_{11}$ is 5 and for $M_{12}$ it is 6.

\end{abstract}

\section{Introduction}

It is natural to try to extend the notion of the dimension of a vector space into group theory.
This gives rise to a number of possible generalizations which all give the same answer in linear algebra but in group theory can give differing results.
Here, we examine two of those possible extensions and make computations about them on the Mathieu groups $M_{11}$ and $M_{12}$.

The primary one we will be examining is $m(G)$, the maximal size of a irredundant generating set of the finite group $G$.
Whiston uses a list of isomorphism classes of maximal subgroups of $S_n$ to give a proof that $m(S_n) = n-1$. \cite{WHISTON}
This method suggests an approach towards computing $m$ for a variety of other groups by using knowledge of their maximal subgroups.
One motivation for these computations is that $m(G)$ specifically gives a bound on the run-time of an algorithm to find a random element of the group.
This algorithm reaches a uniform distribution in time given by $|G|^{O(m(g))} n^2 \log n$ where $G$ is a finite group generated by $n \times n$ matrices.\cite{WHISTON}

This write-up was the author's undergraduate senior thesis under the direction Professor R. Keith Dennis.

\subsection{Mathieu Groups}

The Mathieu groups $M_{11}$, $M_{12}$, $M_{22}$, $M_{23}$, and $M_{24}$ are sporadic simple groups.
They are also interesting for their behavior as transitive group actions.
Suppose that $G$ acts on a set $S$ denoted by $g x$ for $g \in G$ and $x \in S$.
We call the action of $G$ on $S$ \emph{transitive} if for every $x,y \in S$, there is a $g \in G$ such that $g x = y$.
Similarly, we call the action \emph{$k$-transitive} if for every pairwise distinct sequence $(x_1, \ldots, x_k)$ and pairwise distinct sequence $(y_1, \ldots, y_k)$ of  elements in $S$ there is a $g \in G$ so that $(g x_1, \ldots, g x_k) = (y_1, \ldots, y_k)$.
Furthermore, a $k$-transitive action where there is exactly one such $g$ is called \emph{sharply $k$-transitive}.

The simplest examples of $k$-transitive groups are $S_\ell$ and $A_{\ell+2}$ for $\ell \geq k$.
These are generally referred to as \emph{trivial} $k$-transitive groups.
The Mathieu groups are examples of nontrivial $k$-transitive groups with the properties of each listed in Table~\ref{tab:mathieu_properties}.
Each Mathieu group $M_{n}$ acts on a set of size $n$.
It is a consequence of the classification of finite simple groups that the Mathieu groups give the only nontrivial examples of $4$- or $5$-transitive groups. \cite{GROVE}
Since every $k$-transitive group is $j$-transitive for $j < k$, this further says that that there are no nontrivial $k$-transitive groups for $k > 5$.
As such, the Mathieu groups are rather unique.

\begin{table}
    \centering
    \begin{tabular}{c|c|c}
        Mathieu Group & Order & Transitivity \\
        \hline
        $M_{11}$ & 7,920 & sharply $4$-transitive \\
        $M_{12}$ & 95,040 & sharply $5$-transitive \\
        $M_{22}$ & 443,520 & $3$-transitive \\
        $M_{23}$ & 10,200,960 & $4$-transitive \\
        $M_{24}$ & 244,823,040 & $5$-transitive
    \end{tabular}
    \caption{Basic properties of the Mathieu groups.\cite{WIKI}}
    \label{tab:mathieu_properties}
\end{table}

The Mathieu groups may be constructed as one-point extensions.
If a group $G$ acts transitively on a set $S$ of size $n$, then a group $G' \geq G$ acting transitively on $S \cup \{\alpha\}$, for $\alpha \not\in S$, with $\mbox{Stab}_{G'}(\alpha) = G$ is called a \emph{one-point extension} of $G$.
Here we write $\mbox{Stab}_{G}(x) = \{g \in G | gx = x\}$ to mean the set of elements in $G$ which fix $x$.
We start with $M_{10} = A_6 \rtimes \Z_2$, the double cover of $A_6$, which acts 2-transitively on the 10-point projective line $P = \F_9 \cup \{\infty\}$.
$M_{11}$ is a one-point extension of $M_{10}$ and $M_{12}$ is a one-point extension of $M_{11}$.
In particular, $M_{11}$ is a subgroup of $M_{12}$ which is the stabilizer of a point.
Similarly, $M_{24}$ is a one-point extension of $M_{23}$ which is a one-point extension of $M_{22}$, and $M_{22}$ is a one-point extension of $M_{21} = PSL(3,4)$.
It is even possible to start at $M_9 = (\Z_2 \times \Z_2) \rtimes Q_8$ and extend from there, but this group is less frequently used.
Only $M_{11}$, $M_{12}$, $M_{22}$, $M_{23}$, and $M_{24}$ are typically referred to as Mathieu groups and are the only ones of these which are sporadic simple groups, with $M_{10}$ and $M_{9}$ not being simple and $M_{21}$ not being sporadic.
For a more complete description of these groups, refer to Grove, section 2.7.\cite{GROVE}

\section{Background and Definitions}

Suppose that $G$ is a finite group and $\{g_1, \ldots, g_k\}$ is any set of elements in $G$.
Then let $\left<g_1, \ldots, g_k\right>$ be the unique minimal subgroup of $G$ containing $\{g_1, \ldots, g_k\}$.

\begin{definition}
Such a set $\{g_1, \ldots, g_k\}$ is called \emph{generating} if $\left<g_1, \ldots, g_k \right> = G$.
\end{definition}
\begin{definition}
The set is called \emph{irredundant} if
\[ \left<g_1, \ldots, \widehat{g_i}, \ldots, g_k\right> \not= \left<g_1, \ldots, g_k\right> \]
for all $i$ (where $\widehat{g_i}$ denotes that the element $g_i$ is skipped).
\end{definition}
Note that if $s \subseteq G$ is irredundant, then every strict subset $s'$ of $s$ also generates a strict subgroup, i.e., $\left<s'\right> < \left<s\right>$.
Moreover, every subset of an irredundant set is also irredundant.

We are interested primarily in computing the two numbers $m(G)$ and $i(G)$, which we now define.
\begin{definition}
    Let $m(G)$ be the maximal size of an irredundant generating set of a finite group $G$.
\end{definition}
\begin{definition}
    Let $i(G)$ be the maximal size of an irredundant set (not necessarily generating) of a finite group $G$.
\end{definition}
It is immediate that $m(G) \leq i(G)$ for all groups $G$.
Of course, if $G$ is a vector space, then both of these are equal to the dimension of $G$, but for a general group $G$ the two are not necessarily equal.  
For any subgroup $H \leq G$, $i(H) \leq i(G)$ but it is not necessarily true that $m(H) \leq m(G)$.

There is always a subgroup $H \leq G$ so that $m(H) = i(G)$.
To see this, take $\{g_1, \ldots, g_k\}$ an irredundant set in $G$ with $k = i(G)$, then $H = \left< g_1, \ldots, g_k \right>$ satisfies $i(G) \leq m(H) \leq i(H) \leq i(G)$ and so $m(H) = i(G)$.
In particular, this shows that if $i(G) > m(G)$, then there is a subgroup $H \leq G$ so that $m(H) = i(G) > m(G)$.
Furthermore, if $H$ is not $G$ then $H$ is contained in a maximal subgroup $M \leq G$, giving $m(H) \leq i(M) \leq i(G) = m(H)$ and so $m(H) = i(M)$.
Hence $i(G)$ is either equal to $m(G)$ or is equal to the maximum of $i(M)$ over all maximal subgroups $M$ of $G$.
This gives us a useful proposition.
\begin{proposition}
    \label{prop:i_is_max}
    For any finite group $G$,
    \[ i(G) = \max( m(G), \max( i(M) ) ) \]
    where $\max( i(M) )$ is taken over all maximal subgroups $M \leq G$.
\end{proposition}

Another useful fact to note is the following proposition.
\begin{proposition}
    \label{prop:m_leq_i}
    $m(G) \leq \max( i(M) ) + 1$ where $\max( i(M) )$ is taken over all maximal subgroups $M \leq G$.
\end{proposition}
This follows from noting that if $g_1, \ldots, g_k$ is an irredundant generating set of size $m(G)$ then $\left< g_1, \ldots, g_{k-1} \right> \leq M$ for some maximal subgroup $M$ and so $i(M) \geq m(G) - 1$.

There is a slight generalization of these ideas that can be made.
Suppose that $C$ is a conjugacy class of elements in $G$.
Then we may define $m(G,C)$ to be the largest size of an irredundant generating set of $G$ whose elements are all in $C$.
Similarly, we define $i(G,C)$ to be the largest size of an irredunant set of $G$ whose elements are all in $C$.
It is immediate that $m(G,C) \leq m(G)$ and $i(G,C) \leq i(G)$.
These values are not a focus of this paper, but we are able to calculate these as well for some conjugacy classes in $G$.

In the interest of computations, we would like to know $m(G)$ and $i(G)$ in terms of smaller groups.
In the case of direct products, there is a nice result.

\begin{theorem}[Dan Collins and K. Dennis]
    Let $G$ and $H$ be finite groups. Then
    \[ m(G \times H) = m(G) + m(H) \]
    and
    \[ i(G \times H) = i(G) + i(H). \]
\end{theorem}

Furthermore, these values are known for some important groups.
Whiston\cite{WHISTON} gives the following result.
\begin{theorem}[Whiston]
    \[m(S_n) = i(S_n) = n-1 \]
    \[m(A_n) = i(A_n) = n-2 \]
\end{theorem}
Cyclic groups are also known.
\begin{theorem}
    Let $\nu(n)$ be the number of distinct primes dividing $n$. Then
    \[m(\Z_n) = i(\Z_n) = \nu(n). \]
\end{theorem}

We also have a concept of flatness of a group.
\begin{definition}
    A group $G$ is called \emph{flat} if $m(G) = i(G)$.
    $G$ is called \emph{strongly flat} if $G$ is flat and if $s$ is any irredundant set of size $m(G)$ then $s$ generates $G$.
\end{definition}
Note that $G$ flat is equivalent to $m(H) \leq m(G)$ for all subgroups $H \leq G$ and $G$ strongly flat is equivalent to $m(H) < m(G)$ for all proper subgroups $H < G$.

\begin{definition}
Suppose that $s = \{g_1, \ldots, g_k\}$ is an irredundant generating set for $G$.
The set $s$ satisfies the \emph{replacement property} if for every nontrivial $g \in G$ there is an index $1 \leq i \leq k$ so that the set with $g_i$ replaced with $g$ also generates $G$.
If every irredundant generating set of size $k$ satisfies the replacement property, then $G$ is said to satisfy the \emph{replacement property for size $k$}.
\end{definition}

It turns out that if $G$ satisfies the replacement property for size $k$, then $k = m(G)$, so it is convenient to simply say that $G$ satisfies the replacement property.

\subsection{Solvable Groups}

An important class of groups where computation of $m(G)$ is quite feasible is the case where $G$ is solvable.
For any group $G$, we call any sequence $\{e\} = H_0 \leq H_1 \ldots \leq H_k = G$ of maximal length with each $H_i$ normal in $G$ a \emph{chief series} of $G$.
We call a group $G$ \emph{solvable} if there is a sequence of subgroups $\{e\} = H_0 \leq H_1 \ldots \leq H_k = G$ with each $H_i$ normal in $G$ and each $H_{i+1}/H_i$ abelian.
So for every solvable group, $H_0 \leq \ldots \leq H_k$ is a chief series.

We also need one further definition.
\begin{definition}
    For $G$ a group, the Frattini subgroup $\Phi(G)$ is the intersection of all maximal subgroups of $G$.
    If $G$ has no maximal subgroups, then define $\Phi(G) = G$ (for example, if $G$ is trivial or some non-finite groups).
\end{definition}

Then we get a useful result that allows computation of $m(G)$ on many groups that are commonly encountered.
\begin{theorem}[Dan Collins and K. Dennis]
    Let $G$ be a finite solvable group.
    Then $m(G)$ is the number of non-Frattini factors in any chief series for $G$.
\end{theorem}
This result is intuitive since the elements of the Frattini group of $G$ are called ``non-generators'' and can never appear in an irredundant generating set of $G$.

\section{Lower Bounds}

The most direct way to give a lower bound for either $m(G)$ or $i(G)$ is simply to exhibit a set $\{g_1, \ldots, g_k\}$ that satisfies the desired properties.
The simplest method is a brute-force computation checking all the possible size $k$ subsets of $G$.

Another means to exhibit generating sets is by considering a maximal subgroup $M$ of $G$.
If we have an irredundant set $\{g_1, \ldots, g_k\}$ which generates $M$, then adding any additional element $g_0 \in G \setminus M$ to get $\{g_0, g_1, \ldots, g_k\}$ gives a generating set.
This set is not necessarily irredundant.
So to find irredundant generating sets of $G$, we may instead find ones of the smaller group $M$ and check whether extensions are irredundant.

\subsection{Tarksi Extensions}

One method to find irredundant generating sets -- and hence give a lower bound on $m(G)$ -- is referred to as finding `Tarski extensions.'
Suppose that we have an irredundant set $\{g_1, \ldots, g_k\}$.
Then a Tarksi extension of this set is one gotten by replacing some $g_i$ with both $a,b \in G$ where $ab = g$.
That is, the set $\{ a, b, g_2, \ldots, g_k\}$ is a Tarski extension of $\{g_1, \ldots, g_k\}$ if $ab = g_1$.
Certainly this extension generates at least as much as $\{g_1, \ldots, g_k\}$.
If it happens to be irredundant, then we take this to be a successful extension.
This gives a simple, if perhaps inefficient, way to go from short irredundant sets to longer ones.

It is not known exactly which irredundant generating sets have successful Tarksi extensions to other irredundant sets.
In particular, one computation started with a size 2 irredundant generating set and attempted to recursively extend it to a size 5 irredundant set in $M_{11}$ failed to find any extensions beyond size 4.
This suggests that it is not always possible to extend from a minimal size set to a maximal size one.

The idea of these extensions originates from a theorem by Tarski that shows that there is an irredundant, generating subset of $G$ of each size from the smallest possible all the way up to $m(G)$.
Tarksi's proof works by going in the reverse direction of our computations; he shows that for each irredundant generating set of size $k$, $k$ is either the minimal size or there is a set of size $k-1$.

In practice, successful computations were carried out by first guessing the value of $m(G)$ through the techniques in the Upper Bounds section and then brute-force computing sequences of size $m(G)-1$.
Each of these sequences could then be tried for Tarski extensions until one is found.

Furthermore, computations focused on finding sets of elements of order~2.
In order to find these, it is necessary to start with a set $\{g_1, \ldots, g_n\}$, before applying a Tarski extension, where $g_1 = ab$ where $a,b \in G$ are of order 2 and $g_2, \ldots, g_n$ are all of order 2.
Since $a,b$ are of order 2, $\left<a,b\right>$ must be a dihedral group of order $2k$ where $k$ is the order of $ab = g_1$.
Since not all dihedral groups occur within $M_{11}$ or $M_{12}$, we may eliminate many possible sets by restricting the order of $g_1$ only to the orders which arise as dihedral groups.
Specifically, for $M_{11}$ quick computations with GAP show that only $k=2,3,4,5,6$ is possible and for $M_{12}$ only $k = 2,3,4,5,6,8,10$ is possible.
Reducing the search to only consider order two elements makes this problem feasible.
Sequences for $M_{11}$, with only 7,920 elements are quick to find even while looking through all possible generators, but $M_{12}$, having 95,040 elements but only 891 order 2 elements, proved too large to find any sets without using this simplification.
While this simplification is not guaranteed to work, computations without it took too long.

We also know some results about when Tarksi extensions can be successful.
\begin{proposition}
    Suppose $s = \{g_1, \ldots, g_k\}$ is an irredundant generating set of size $k$, and define $H_i = \left<g_1, \ldots, \widehat g_i, \ldots, g_k\right>$ generated by size $k-1$ subsets of $s$.
    If $H_i$ is contained in a unique maximal subgroup $M_i$ of $G$, then a Tarski extension replacing $g_i$ will fail to be irredundant.
\end{proposition}

\begin{proof}
    Suppose that $s = \{g_1, \ldots, g_k\}$ extends to an irredundant generating set $\{a,b, g_2, \ldots, g_k\}$.
    Then $\left<a,g_2, \ldots, g_k\right> \leq M_1$ and $\left<b, g_2, \ldots, g_k\right> \leq M_2$ for maximal subgroups $M_1, M_2 \subset G$.
    Since $M_1$ is maximal and $\left<a,b,g_2, \ldots, g_k\right>$ generates $G$, $b \not\in M_1$.
    Similarly, $a \not\in M_2$, and so $M_1$ and $M_2$ are distinct maximal subgroups.
    Hence $\left<g_2, \ldots, g_k\right>$ is contained in both $M_1$ and $M_2$.
\end{proof}
In particular, this means that if $\left<g_1, \ldots, \widehat g_i, \ldots, g_k\right>$ is maximal then we cannot successfully Tarksi extend $g_i$.
This gives a simple test to avoid trying hopeless Tarksi extensions.

\subsection{Computed Sequences}

Computations yielded the following result.

\begin{proposition}
    \label{prop:lower_bound}
    $m(M_{11}) \geq 5$ and $m(M_{12}) \geq 6$.
\end{proposition}
\begin{proof}
    These lower bounds are given by explicit irredundant generating sequences.
    Refer to Appendix A for the computations which produced these.
\end{proof}

One explicit, size 5 irredundant generating set of $M_{11}$ is given by the permutations
\begin{align*}
(4,10)(5,8)(6,7)(9,11) \\
(3,4)(5,7)(6,9)(8,11) \\
(3,5)(4,6)(7,9)(10,11) \\
(2,10)(3,11)(4,8)(6,9) \\
(1,3)(4,8)(5,10)(6,7)
\end{align*}
Here we are treating $M_{11}$ as a permutation group given by
\[ M_{11} = \left<  (1,2,3,4,5,6,7,8,9,10,11),\; (3,7,11,8)(4,10,5,6) \right>. \]
Similarly, a size 6 irredundant generating set of $M_{12}$ is
\begin{align*}
    (5,7)(6,11)(8,9)(10,12) \\
    (4,5)(6,12)(8,11)(9,10) \\
    (4,6)(5,10)(7,8)(9,12) \\
    (3,7)(4,8)(5,11)(9,10) \\
    (1,4)(5,11)(6,7)(10,12) \\
    (2,11)(4,8)(6,7)(9,12)
\end{align*}
where $M_{12}$ is given as a permutation group as
\begin{align*} M_{12} = \big< &(1,2,3,4,5,6,7,8,9,10,11), \; (3,7,11,8)(4,10,5,6),  \\
                              &\qquad(1,12)(2,11)(3,6)(4,8)(5,9)(7,10) \big>.
\end{align*}

Note that all of the generating elements are of order 2.
So these generating sets also answer two cases of the question, ``Is it possible to find, for $S$ a finite simple group, an irredundant generating set of size $m(S)$ whose elements are all of order 2?''
Furthermore, it is interesting to observe that every size $m-1$ subset of these elements generates a maximal subgroup.
For the given set generating $M_{11}$ the isomorphism classes of these subgroups are $(PSL(2,11), PSL(2,11), PSL(2,11), A_6, A_6)$, ordered by the $i$-th maximal subgroup being generated by all but the $i$-th generating element.
For $M_{12}$, all size 5 subsets of the given generating set generate copies of $M_{11}$.
So far, every maximal size irredundant generating set found for  $M_{11}$ and $M_{12}$ has also had every size $m-1$ subset generating maximal subgroups.

Not only are all the generating elements in these computed sets of order two, but all the elements are in the same conjugacy class.
In particular, the first set gives $m(M_{11}, 2A) \geq 5$, where $2A$ is the sole conjugacy class of order 2 elements in $M_{11}$.
The second set gives $m(M_{12}, 2B) \geq 6$, where $2B$ is the second conjugacy class of order 2 elements in $M_{12}$, which contains 495 elements.
Here the notation $2A$ and $2B$ refers to the conjugacy classes listed by the ATLAS of Finite Groups. \cite{ATLAS}

\section{Upper Bounds}

The primary means of constructing upper bounds on $m(G)$ is by giving an upper bound on $i(G)$ in terms of subgroups of $G$.
This is the method used by Whiston in his computation of $m(S_n) = i(S_n) = n-1$ and $m(A_n) = i(A_n) = n-2$.
First, we examine the maximal subgroups of $M_{11}$ and $M_{12}$.
Many of these subgroups we may compute with directly based upon their structure.
Of particular importance is the ability to compute $m$ and $i$ of solvable groups, which covers almost all the maximal subgroups of $M_{11}$ and $M_{12}$.
For those groups that cannot be computed directly, we then repeat this process by looking at their maximal subgroups and computing on those.
At each of these steps we apply Proposition~\ref{prop:i_is_max} to get an upper bound on both $m(G)$ and $i(G)$.

\subsection{Maximal Dimension}

In order to give an upper bound on $m(G)$, we will look at the maximum dimension of $G$.

\begin{definition}
Let $S = \{M_1, \ldots, M_k\}$ with each $M_i \leq G$.
$S$ is said to be in \emph{general position} for all $I \subsetneq J \subseteq \{1, 2, \ldots, k\}$,
\[ \bigcap _{i \in I} M_i \supsetneq \bigcap _{j \in J} M_j \]
\end{definition}

\begin{definition}
For $G$ a finite group, the \emph{maximum dimension} of $G$, denoted \\* $\mbox{MaxDim}(G)$, is defined to be
\[ \max \left\{ \left| S \right| : S \mbox{ in general position with each $M_i \in S$ maximal in $G$} \right\}. \]
\end{definition}

We now note that there is a relationship between $m(G)$ and $\mbox{MaxDim}(G)$.

\begin{proposition}
Let $G$ be a group and $g=\left\{g_1, g_2, \ldots, g_n\right\}$ an irredundant, generating set of $G$.
Each size $n-1$ subset of $G$ given by $\{g_1, \ldots, \widehat{g_{i}}, \ldots, g_n\}$ generates a proper subgroup of $G$ and so is contained in at least one maximal subgroup $K_i$ of G.
Let $K = \left\{K_1, K_2, ..., K_n \right\}$.
$K$ is in general position.
\end{proposition}

\begin{proof}
Suppose that $I \subsetneq J \subseteq \{1, 2, \ldots, n\}$.
Take $j \in J$ with $j \not\in I$.
Then $g_j \in \bigcap_{i \in I} K_i$.
Since $K_j$ contains $\{g_1, \ldots, \widehat{g_j}, \ldots, g_n\}$ and is a maximal subgroup of $G$, it cannot contain the final generator $g_j$.
So $g_j \not\in \bigcap_{j \in J} K_j$.
Hence, $K$ is in general position.

\end{proof}

\begin{corollary}
$m(G) \leq \mbox{MaxDim}(G) \leq i(G)$
\end{corollary}

The above proposition also gives another useful result by noting that each $K_i$  contains an irredundant set $\{g_1, \ldots, \widehat{g_i}, \ldots, g_n\}$.
\begin{corollary}
    \label{corollary:num_max_sg}
    There must be at least $m(G)$ distinct, maximal subgroups  \\*
    $\{M_1, \ldots, M_{m(G)}\}$ of $G$ with $i(M_j) \geq m(G) - 1$.
\end{corollary} 
This result will be used to give an upper bound on $m$ of groups by examining their maximal subgroups.
Note that this is a stronger result than that given in Proposition~\ref{prop:m_leq_i}.

We now also give some additional results building off of the ideas from MaxDim.
Let $\mbox{Max}(G)$ denote the set of all maximal subgroups of $G$.
Let $\mbox{Allow}(G)$ denote the set of all maximal subgroups of $G$ that occur in a family of maximal subgroups which is in general position and is associated to an irredundant generating set of size $m(G)$.

\begin{proposition}
    \label{prop:MaxDim_ineqs}
    Let $G$ be a finite group. Then either
    \begin{enumerate}
        \item $m(G) = \mbox{MaxDim}(G) \leq 1 + \max\{i(M) | M \in \mbox{Allow}(G) \}$, or
        \item $m(G) < \mbox{MaxDim}(G) \leq \max\{i(M) | M \in \mbox{Max}(G) \}.$
    \end{enumerate}
    Furthermore, if $\max\{i(M) | M \in \mbox{Max}(G) \} \leq m(G)$, then $G$ is flat. \\*
    If $\max\{i(M) | M \in \mbox{Max}(G) \} < m(G)$, then $G$ is strongly flat.
\end{proposition}

\begin{proof}
    Suppose $F = \{M_1, \ldots, M_k\}$ with $|F| = \mbox{MaxDim}(G) = k$ is in general position.
    Then there is a $g_j \in \bigcap_{i \not= j} M_i$ such that $g_j \not\in \bigcap_{i} M_i$.
    So $s = \{g_1, \ldots, g_k\}$ is an irredundant set.
    If $\left<s\right> = G$, then we have the first case.
    Otherwise, $\left<s\right>$ lies inside a maximal subgroup $M < G$.
    So $\mbox{MaxDim}(G) = k \leq i(M)$ for this $M$, which gives the second case.

    The last two statements follow directly from noting that $m(H) \leq i(M)$, where $H \leq G$ is any proper subgroup of $G$ and $M$ is any maximal subgroup containing $H$.
\end{proof}

\begin{proposition}
    \label{gen_pos_ineqs}
    Let $F = \{H_1, \ldots, H_k\}$ be a family of subgroups in general position.
    For any subset $I \subseteq \{1, \ldots, k\}$, we have
    \[ k \leq i(H_I) + |I| \]
    where $H_I = \bigcap_{j \in I} H_j$.
    In particular, for each integer $\ell$, $1 \leq \ell \leq k$, we have
    \[ k \leq i(H_\ell) + 1. \]
\end{proposition}

\begin{proof}
    Let $J_1 = I_1 \sqcup I \subseteq \{1, \ldots, k\}$ and $J_2 = I_2 \sqcup I \subseteq \{1, \ldots, k\}$
    be disjiont unions with $J_2 \subset J_1$ a proper containment (so $I_2 \subset I_1$ is a proper containment).
    Then, since $F$ is in general position, $H_{J_1} \subsetneq H_{J_2}$.
    Hence
    \[ \{H_i \cap H_I | i \not\in I \} \]
    is a collection of $k - |I|$ proper subgroups of $H_I$ in general position.
    Thus
    \[ k - |I| \leq i(H_I). \]
\end{proof}

We next give a characterisation of the sets in $G$ that satisfy the replacement property.
Define $\mbox{rad}(F) = \bigcap_i H_i$ for $F = \{H_1, \ldots, H_k\}$ a family of subgroups in general position.
\begin{proposition}
    Let $s$ be an irredundant generating sequence of size $k$ for the finite group $G$.
    Then $s$ satisfies the replacement property for $G$ if and only if $\mbox{rad}(F) = 1$ for every family $F$ of maximal subgroups in general position that is associated to $s$.
\end{proposition}

This allows us to get that $G$ satisfies the replacement property from its maximal subgroups.
\begin{proposition}
    Let $G$ be a finite subgroup, $m = m(G)$ and $s = \{g_1, \ldots, g_m\}$ an irredundant generating set of $G$ size $m$.
    Let $F = \{M_1, \ldots, M_m\}$ be an associated family of maximal subgroups in general position.
    Assume that for any such $F$, there exists one of the maximal subgroups, say $M_m$, such that
    \begin{enumerate}
        \item $M_m = \left< g_1, \ldots, g_{m-1} \right>$
        \item $m(M_m) = m - 1$
        \item $M_m$ satisfies the replacement property
    \end{enumerate}
    Then $G$ satisfies the replacement property.
\end{proposition}

\begin{proof}
    Note that for $j \not= m$ we have
    \begin{enumerate}[(a)]
        \item $M_m \cap M_j \geq \left< g_1, \ldots, \widehat{g_j}, \ldots, \widehat{g_m}, \ldots, g_k\right>$
        \item $M_m \cap M_j \not= M_m$ since $F$ is in general position
        \item Thus there exists $N_j \in \mbox{Max}(M_m)$ with $N_j \geq M_m \cap M_j$
        \item Hence $F' = \{N_1, \ldots, N_{m-1}\}$ is a family of maximal subgroups of $M_m$ in general position associated to the irredundant generating set $s' = \{g_1, \ldots, g_{m-1}\}$.
        \item Since $M_m$ satisfies the replcament property, we have 
            \[\mbox{rad}(F') = N_1 \cap \cdots \cap N_{m-1} = 1 \]
    \end{enumerate}
    Thus
    \begin{align*}
        \mbox{rad}(F') &= N_1 \cap \cdots \cap N_{m_1} \\
                     &\geq (M_m \cap M_1) \cap \cdots \cap (M_m \cap M_{m-1}) \\
                     &= M_1 \cap M_2 \cap \cdots \cap M_{m-1} \cap M_m \\
                     &= \mbox{rad}(F).
    \end{align*}
    and since $\mbox{rad}(F') = 1$ we have $\mbox{rad}(F) = 1$ as well.
\end{proof}

\subsection{Computed Upper Bounds}

\begin{proposition}
    \label{prop:upper_bound}
    $m(M_{11}) \leq 5$ and $m(M_{12}) \leq 6$.
\end{proposition}
\begin{proof}
 Computations based upon applying Corollary~\ref{corollary:num_max_sg}.
\end{proof}
Tables~\ref{tab:M11} and~\ref{tab:M12} list the maximal subgroups of $M_{11}$ and $M_{12}$, respectively, and give the computed values for $i$ and $m$, when available.
In particular, in $M_{12}$ we have three maximal subgroups that are not solvable, one being $M_{11}$, one other $S_6 \rtimes \Z_2 \approx \mbox{Aut}(S_6)$, and the last being $PSL(2,11)$.
In $M_{11}$, the maximal subgroup $M_{10} = A_6 \rtimes \Z_2$ also arises in $\mbox{Aut}(S_6)$ and so is included in the calculation for $M_{12}$.
For $PSL(2,11)$, the result $m=i=4$ has been calculated by exhaustive searches of generating sets.\cite{NACHMAN}

\begin{table}
    \centering
    \begin{tabular}{r|c|c|c|c}
        Maximal Subgroup & $m$ & $i$ & Solvable & ID\\
        \hline
        $M_{10}$ &   & 4 & No & [720, 765]\\
        $PSL(2,11)$ & 4 & 4 & No & [660,13]\\
        $S_5$ & 4 & 4 & No & [120,34]\\
        $M_9 \rtimes \Z_2$ &  & 4 & Yes & [144,182]\\
        $Q_8 \rtimes S_3 \approx GL(2,3)$ &  & 3 & Yes & [48,29]
    \end{tabular}
    \caption{Maximal subgroups of $M_{11}$.
             Values for $m$ are listed when available but only $i$ was necessary for the computation so not all were computed.
             The ID of a group refers to its `small group ID' given by the GAP programming environment.
            This is a unique classification of the isomorphism class of a group for most groups of order less than 2000.
            }
    \label{tab:M11}
\end{table}

\begin{table}
    \centering
    \begin{tabular}{r|c|c|c|c|c}
        Maximal Subgroup & $m$ & $i$ & Count & Solvable & ID\\
        \hline
        $M_{11}$ &  5 & 5 & 24 & No & - \\
        $\mbox{Aut}(S_6)$ &  & 5 & 132 & No & [1440, 5841] \\
        $PSL(2,11)$ & 4 & 4 & 144 & No & [660, 13]\\
        $\Z_3^2 \rtimes (2 . S_3)$ &  & 4 & 440 & Yes & [432, 734] \\
        $S_5 \times \Z_2$ & 5 & 5 & 396 & No & [240,189] \\
        $Q_8 \rtimes S_4$ &  & 4 & 495 & Yes & [192,1494]  \\
        $\Z_4^2 \rtimes (\Z_2 \times S_3)$ &  & 4 & 495 &  Yes & [192,956] \\
        $A_4 \times S_3$ & 4 & 4 & 1320 & Yes & [72, 44]\\
    \end{tabular}
    \caption{Maximal subgroups of $M_{12}$.
             Count gives the number of maximal subgroups in the isomorphism class.
             $2. S_3$ denotes the double cover of $S_3$.
            }
    \label{tab:M12}
\end{table}

\begin{table}
    \centering
    \begin{tabular}{r|c|c|c|c}
        Maximal Subgroup & $m$ & $i$ & Count & ID \\
        \hline
        $(\Z_2 \times D_8) \rtimes \Z_2$ & 3 & 3 & 45 & [32,43]\\
        $\Z_2 \times (\Z_5 \rtimes \Z_4)$ & 3 & 3 & 36 & [40,12]\\
        $((\Z_3 \times \Z_3) \rtimes \Z_8) \rtimes \Z_2$ & 3 & 4 & 10 & [144,182]\\
        $M_{10}$ & 4 & 4 & 1 & [720,763]\\
        $A_6 \rtimes \Z_2$ & 4 & 4 & 1 & [720,765]\\
        $S_6$ & 5 & 5 & 1 & [720,764]\\
    \end{tabular}
    \caption{Maximal subgroups of $\mbox{Aut}(S_6)$.}
    \label{tab:AutS6}
\end{table}

\begin{table}
    \centering
    \begin{tabular}{r|c|c|c|c}
        Maximal Subgroup & $m$ & $i$ & Count & ID\\
        \hline
        $D_{16}$ & 2 & 2 & 45 & [16,7]\\
        $D_{20}$ & 2 & 2 & 36 & [20,4]\\
        $\Z_3^2 \rtimes \Z_8$ & 2 & 3 & 10 & [72,39] \\
        $A_6$ & 4 & 4 & 1 & [360,118]\\
    \end{tabular}
    \caption{Maximal subgroups of $A_6 \rtimes \Z_2$.}
    \label{tab:A6rtimesZ2}
\end{table}

\begin{table}
    \centering
    \begin{tabular}{r|c|c|c|c}
        Maximal Subgroup & $m$ & $i$ & Count & ID\\
        \hline
        $QD_{16}$ & 2 & 2 & 45 & [16,8] \\
        $\Z_5 \rtimes \Z_4$ & 2 & 2 & 36 & [20,3] \\
        $M_9$ & 3 & 3 & 10 & [72,41] \\
        $A_6$ & 4 & 4 & 1 & [360,118] \\
    \end{tabular}
    \caption{Maximal subgroups of $M_{10}$.}
    \label{tab:M10}
\end{table}

For the computation of $\mbox{Aut}(S_6)$, first note that $i(\mbox{Aut}(S_6)) \geq 5$ since $S_6 \leq \mbox{Aut}(S_6)$.
Furthermore, no other maximal subgroup of $\mbox{Aut}(S_6)$ has an $i$ value greater than 4.
Since there is only one subgroup of $\mbox{Aut}(S_6)$ isomorphic to $S_6$, Corollary~\ref{corollary:num_max_sg} gives us that $m(G)$ must be at most 5.
We then get that $i(\mbox{Aut}(S_6)) = 5$ by Proposition~\ref{prop:i_is_max}.

The two maximal subgroups $M_{10}$ and $A_6 \rtimes \Z_2$ are also computed using this technique of listing maximal subgroups.
Their results are listed in Tables~\ref{tab:A6rtimesZ2} and~\ref{tab:M10}.
By the same argument as for $\mbox{Aut}(S_6)$, we get $i(M_{10}) = 4$ and $i(A_6 \rtimes \Z_2) = 4$ as there is only one maximal subgroup in each with $i$ at least $4$.

Applying Proposition~\ref{prop:i_is_max} to this result then gives a slightly stronger result.
\begin{corollary}
    \label{corollary:upper_bound_i}
    $i(M_{11}) \leq 5$ and $i(M_{12}) \leq 6$.
\end{corollary}
\section{Conclusion}

Computational results gave, in Proposition~\ref{prop:lower_bound}, explicit irredundant generating sets of $M_{11}$ and $M_{12}$ of sizes 5 and 6, respectively.
Consideration of the maximal subgroups of $M_{11}$ and $M_{12}$ give, in Corollary~\ref{corollary:upper_bound_i}, upper bounds of 5 and 6, respectively, on both $m$ and $i$ by computing $i(M)$ for each maximal subgroup of $M$.
Together these give the full result.
\begin{theorem}
    We have that $m(M_{11}) = i(M_{11}) = 5$ and $m(M_{12}) = i(M_{12}) = 6$.
    Moreover, $M_{11}$ and $M_{12}$ are both strongly flat.
\end{theorem}
By noting which conjugacy classes the given generating elements are from, we get another result.
\begin{proposition}
    We have that $m(M_{11},2A) = i(M_{11},2A) = 5$ and $m(M_{12}, 2B) = i(M_{12}, 2B) = 6$.
\end{proposition}

The same techniques used here should extend to the remaining Mathieu groups, but the increase in size of these groups has so far made computations infeasible.
The next Mathieu group, $M_{22}$, has order 443,520 making it almost 5 times as large as $M_{12}$.
More optimistically, $M_{22}$ only has 1,150 elements of order 2 which is comparable to the 891 elements in $M_{12}$ of order 2, as was done for $M_{11}$ and $M_{12}$.
It may then be possible to find a maximal size irredundant generating set of $M_{22}$ composed of order 2 elements.

We suspect that $m(M_{22}) = 6$, $m(M_{23}) = 7$ and $m(M_{24}) = 8$ by considerations of their maximal subgroups and hope to be able to carry out the full computation on those groups as well.

\section{Acknowledgments}

Professor Keith Dennis was the primary driving force behind these ideas and many of the computations and provided invaluable feedback.
\pagebreak

\appendix
\section{Appendix - Code}
The following is the GAP code used for the calculation of irredundant generating sets.

\input{code}

\section{Appendix - Utility Code}
The following are several utility scripts used by the other code to assist in things such as output.
\input{util}

\end{document}

%% file: code.tex
\begin{verbatim}
IsIrredundant := function(G,seq)
    # Test whether a given sequence is irredundant 
    # and generates G
    local irred;
    irred := IsIrredundantNC(G,seq);
    if irred and Size(Subgroup(G,seq)) = Size(G) then
        return true;
    fi;
    return false;
end;

IsIrredundantNC := function(G,seq)
  # Test whether a given sequence is irredundant
  # G is the group the sequence is from 
  # This assumes that seq generates G
  local l,sg,i,new;
  l := Length(seq);
  sg := Size(G);

  for i in [1..l] do
    new := ShallowCopy(seq);
    Remove(new, i);
    if(Size(Subgroup(G,new))=sg) then
      return false;
    fi;
  od;
  return true;
end;

IdSubSeq:=function(G,seq)
    # Prints a nice description of the given sequence
    # IDs and descriptions of groups generated
    # by the length l-1 subsequences
    local l,sg,i,new,h;
    l:=Length(seq);
    sg:=Size(G);

    for i in [1..l] do
        new:=ShallowCopy(seq);
        Remove(new,i);
        h:=Subgroup(G,new);
        Print(SafeIdGroup(h),"\t   ",NiceStructureDescription(h),"\n");
    od;
end;

TarskiExtTwo:=function(G,seq)
    local ll,ord,cc,two,x,a,b,new;
    # assumes seq generates G
    # quits when it finds one
    # return true if found, false if not found
    # structure of sequence:
    # seq = (g1,g3,...,g{m-1})
    # o(g{m-1}) in dih[k]; o(gi) = 2 for i < m-1
    # ONLY factor last term into a product of
    # elements of order 2
    ll:=Length(seq);
    ord:=2;
    cc:=ConjugacyClasses(G);
    cc:=Filtered(cc,x->(ord=Order(Representative(x))));
    two:=Flat(List(cc,Elements));
    x:=seq[ll];
    for a in two do
        new:=ShallowCopy(seq);
        Remove(new,ll);
        b:=a*x;
        if(Order(b)=2) then
            Append(new,[a,b]);
            if(IsIrredundantNC(G,new)) then
                Print(CleanOut(new),"\n");
                return true;
            fi;
        fi;
    od;
    return false;
end;

TestMathieuTarskiNew:=function(k,n,rand)
    # Attempt to find length k+1 irredundant generating sequences
    # for the n-th Mathieu group
    # rand=true to try sequences in a random order
    local ord,g,cc,cd,c,two,ll,seq,gseq,test,i,h,possible,cblln,time0,time,
      status,dih,DIH,ggseq,dd,line1,line2,three;

    time0:=Runtime();
    possible:=[11,12,22,23,24];
    if(not k in possible) then
        Print(k," must be in ",possible,"\n");
        return " ";
    fi;

    line1:= "------------------------------------------------";
    line2:= "================================================";

    g:=MathieuGroup(k);
    cd:=ConjugacyClasses(g);
    # List of possible 'dihedral' orders in each Mathieu group
    dih:=[];
    dih[11]:=[2,3,4,5,6];
    dih[12]:=[2,3,4,5,6,8,10];
    dih[22]:=[2,3,4,5,6];
    dih[23]:=[2,3,4,5,6];
    dih[24]:=[2,3,4,5,6,8,10,11,12];
    # find elements of "dihedral" order
    cd:=Filtered(cd,x->(Order(Representative(x)) in dih[k]));
    cc:=ShallowCopy(cd);
    DIH:=Flat(List(cd,Elements));
    # find elements of order 2
    cc:=Filtered(cc,x->(2=Order(Representative(x))));
    two:=Flat(List(cc,Elements));
    # find elements of order 3
    cc:=ShallowCopy(cd);
    cc:=Filtered(cc,x->(3=Order(Representative(x))));
    three:=Flat(List(cc,Elements));
    if(rand) then
        two:=FisherYates(two);
    fi;
    ll:=Size(two);
    cblln:=Commify(Binomial(ll,n-1));
    Print(k,"  ",ll,"  ",cblln,"\n");
    seq:=InitComb(ll,n-1);
    test:=true;

    # Cycle through length n-1 sequences
    # If they are
    while(test) do
        if(0=seq[n] mod 5000) then
            time:=FormatTime(Runtime()-time0);
            Print(RJust(15,Commify(seq[n])),
                    "/",cblln,"  M",k,"  ",time,"\n");
        fi;

        # Construct our new sequence of n-1 order 2 elements
        gseq:=[];
        for i in [1..n-1] do
            Append(gseq,[two[seq[i]]]);
        od;

        # Append a new element of 'dihedral' order
        for dd in DIH do
        # It can often work to just use order 3 elements
        #for dd in three do
            ggseq:=ShallowCopy(gseq);
            Append(ggseq,[dd]);
            # Test to see whether it is irred and generating
            if(IsIrredundant(g,ggseq)) then
                Print("Found irr gen seq of length ",n,":\n");
                Print(CleanOut(gseq),"\n");
                Print(CleanOut(ggseq),"\n");
                Print("is irredundant!\n");
                Print("ord dih element = ",Order(ggseq[n]),"\n");

                # Print sequence info
                IdSubSeq(g,ggseq);

                # insert Tarski extension
                Print("Applying Tarski extension test:\n");
                status:=TarskiExtTwo(g,ggseq);
                test:=not status;
                if(test) then
                    Print("Nothing found, continuing ...\n");
                fi;
                Print(line1,"\n");
            fi;
        od;

        ## end, append "dihedral" element
        if(seq[n]=seq[n+1]) then
            test:=false;
            break;
        fi;
        seq:=NextCombination(ll,n-1,seq);
    od;
    if(not test) then
        Print(line2,"\n");
        Print("Found irr gen seq of length ",1+n," for M",k,"\n");
        Print(line2,"\n");
    fi;
end;

DihedralInMathieu:=function(k)
    # Determines for which n there are dihedral groups of order 2n
    # in the Mathieu group M_k
    local g,ord,div,d,ll,possible;
    possible:=[11,12,22,23,24,25];
    if(not k in possible) then
        Print(k," must be in ",possible,"\n");
        return " ";
    fi;
    g:=MathieuGroup(k);
    ord:=Size(g);
    Print(ord,"  ",CleanOut(Factors(ord)),"\n");
    div:=Divisors(ord);
    for d in div do
        if(0 = d mod 2) then
            ll:=IsomorphicSubgroups(g,DihedralGroup(d));
            # Print(d,"\n");
            if(Size(ll)>0) then
                Print("yes:  ",d,"\n");
            fi;
       fi;
   od;
end;
\end{verbatim}

%% file: util.tex
\begin{verbatim}
Commify:=function(n)
    local quo,cnt,out,rem;
    if(not n in Integers) then
        Print(n," is not an integer\n");
        break;
    fi;
    quo:=n;
    cnt:=0;
    out:=[];
    while(true) do
        rem:=quo mod 10;
        out:=Concatenation(String(rem),out);
        quo:=(quo-rem)/10;
        cnt:=1+cnt;
        if(quo=0) then
            return out;
        fi;
        if(0=cnt mod 3) then
            out:=Concatenation(",",out);
        fi;
    od;
end;

SafeIdGroup:=function(g)
    local  n,gid,spc;
    # don't crash GAP if group can't be identified
    n:=Size(g);
    if(n=1024 or n=512 or n=1536) then
        return false;
    fi;
    # not sure if IdGroup can handle anything if |G|>2000
    if(n>2000) then
        return false;
    fi;
    if(SMALL_AVAILABLE(n) <> fail) then
        gid:=String(IdGroup(g));
        spc:=String(' ');
        # remove spaces
        RemoveCharacters(gid,spc);
        return gid;
    else
        return false;
    fi;
end;

CleanOut:=function(x)
    local str,spc,i,tmp,sqt,dqt;
    # remove spaces, ' and " from output
    str:=String(x);
    spc:=String(' ');
    sqt:=String("'");
    dqt:=String('"');
    # remove spaces
    RemoveCharacters(str,spc);
    # remove single quotes
    RemoveCharacters(str,sqt);
    # remove double quotes
    RemoveCharacters(str,dqt);
    return str; 
end;

NiceStructureDescription:=function(g)
    local str,spc;
    # remove spaces from StructureDescription
    str:=String(StructureDescription(g));
    spc:=String(' ');
    # remove spaces
    RemoveCharacters(str,spc);
    return str;
end;

RJust:=function(n,x)
    local str,len,pad,spc,add,i,tmp,sqt;
    # right justify to fill n characters when printing
    str:=String(x);
    spc:=String(" ");
    sqt:=String("'");
    spc:=List(spc,String);
    len:=Length(str);
    pad:=ShallowCopy(spc[1]);
    add:=n-len-1;
    if(n>len) then
       for i in [1..add] do
           tmp:=ShallowCopy(spc[1]);
           Append(pad,tmp);
       od;
       # remove single quotes
       RemoveCharacters(pad,sqt);
       Append(pad,str);
       str:=pad;    
    fi;
    return str; 
end;

FormatTime := function(time)
    local time2;
    time2 := "";
    time2:=Concatenation(time2,RJust(3,
                String(QuoInt(time,24*60*60*1000))),"d");
    time:=time - QuoInt(time,24*60*60*1000)*24*60*60*1000;
    time2:=Concatenation(time2,RJust(2,
                String(QuoInt(time,60*60*1000))),"h");
    time:=time - QuoInt(time,60*60*1000)*60*60*1000;
    time2:=Concatenation(time2,RJust(2,
                String(QuoInt(time,60*1000))),"m");
    time:=time - QuoInt(time,60*1000)*60*1000;
    time2:=Concatenation(time2,RJust(2,
                String(QuoInt(time,1000))),"s");
    time:=time - QuoInt(time,1000)*1000;
    time2:=Concatenation(time2,RJust(3,String(time)),"ms");
    return time2;
end;

## returns the list of numbers that are divisors of n
Divisors := n -> Filtered([1..n],i->(n mod i) = 0);

#----------------------------------------------------------------------
# The following functions are used for iterating through
# possible length n sequences

InitComb:=function(m,n)
    local B,seq;
    B:=Binomial(m,n);
    seq:=[1..n];
    Append(seq,[1]);
    Append(seq,[B]);
    return seq;
end;

NextCombination := function(m,n,seq)
    local i,j;
    if(seq[n+1]<=seq[n+2]) then    
        seq[n+1]:=1+seq[n+1];
        for i in [n,n-1..1] do
            if(seq[i]<m-(n-i)) then
                seq[i]:=1+seq[i];
                for j in [i+1..n] do
                    seq[j]:=1+seq[j-1];
                od;
                return seq;
            fi;
        od;
    fi;
end;

ListComb:=function(m,n)
    local seq,test;
    seq:=InitComb(m,n);
    test:=true;
    while(test) do
        Print(seq,"\n");
        if(seq[n+1]=seq[n+2]) then
            test:=false;
            continue;
        fi;
        seq:=NextCombination(m,n,seq);
    od;
end;

FisherYates:=function(seq)
    local i,j,l,t;
    #  Fisher-Yates shuffle:
    #  generate a random permutation of array in place
    l:=Length(seq);
    for i in [l,l-1..1] do
        j:=Random(1,i);
        t:=seq[i];
        seq[i]:=seq[j];
        seq[j]:=t;
    od;
    return seq;
end;
\end{verbatim}

%% file: writeup.bbl
\begin{thebibliography}{10}
    \bibitem{ATLAS} Conway, John Horton; Parker, Richard A.; Norton, Simon P.; Curtis, R. T.; Wilson, Robert A. (1985) ATLAS of Finite Groups. Oxford University Press, ISBN 978-0-19-853199-9, MR827219.
    \bibitem{GAP} The GAP Group, GAP -- Groups, Algorithms, and Programming, Version 4.6.4; 2013. (http://www.gap-system.org)
    \bibitem{GROVE} Grove, Larry. \emph{Groups and Characters.} John Wiley \& Sons, Inc. 1997.
    \bibitem{WIKI} \emph{Mathieu Groups.} Wikiepdia. 2013.
    \bibitem{NACHMAN} Nachman, Benjamin. \emph{Generating Sequences of $PSL(2,p)$.} arXiv:1210.2073
    \bibitem{WHISTON} Whiston, Julius. \emph{Maximal Independent Generating Sets of the Symmetric Group.} Journal of Algebra. 2000.

\end{thebibliography}
